\documentclass{amsart}
\usepackage{amsmath,amssymb,amsfonts,latexsym}
\usepackage[latin1]{inputenc}
\usepackage{amsthm}
\usepackage{verbatim}
\usepackage{array}
\usepackage{graphicx,color}
\usepackage{graphics}
\usepackage{epsfig}
\usepackage[normalem]{ulem}

\theoremstyle{plain}
\newtheorem{theorem}{Theorem}[section]
\newtheorem{lemma}[theorem]{Lemma}

\theoremstyle{definition}

\newtheorem{remark}[theorem]{Remark}

\numberwithin{equation}{section}

\newcommand{\ds}{\displaystyle}
\newcommand{\dint}{\ds\int}

\newcommand{\dx}[1]{\; {\rm d}#1}
\newcommand{\eqskip}{ \vspace*{2mm}\\ }
\newcommand{\R}{\mathbb{R}}
\newcommand{\N}{\mathbb{N}}
\newcommand{\Z}{\mathbb{Z}}

\DeclareMathOperator{\spn}{span}

\def\Xint#1{\mathchoice
{\XXint\displaystyle\textstyle{#1}}%
{\XXint\textstyle\scriptstyle{#1}}%
{\XXint\scriptstyle\scriptscriptstyle{#1}}%
{\XXint\scriptscriptstyle\scriptscriptstyle{#1}}%
\!\int}
\def\XXint#1#2#3{{\setbox0=\hbox{$#1{#2#3}{\int}$ }
\vcenter{\hbox{$#2#3$ }}\kern-.6\wd0}}

\def\dashint{\Xint-}

\begin{document}

\title[Summation formulae for Schr\"odinger operators]{Summation formula inequalities for eigenvalues of Schr\"odinger operators}

%\thanks{J.B.K. was supported by a grant within the scope of FCT's project PTDC/\ MAT/\ 101007/2008 and by a fellowship of the
%Alexander von Humboldt Foundation, Germany. Both authors were partially supported by FCT's projects PTDC/\ MAT/\ 101007/2008 and
%PEst-OE/MAT/UI0208/2011}

\author[P. Freitas and J.B. Kennedy]{Pedro Freitas \and James B. Kennedy}

\address{Department of Mathematics, Faculty of Human Kinetics {\rm and}
 Group of Mathematical Physics, Universidade de Lisboa, Complexo Interdisciplinar, Av.~Prof.~Gama Pinto~2,
 P-1649-003 Lisboa, Portugal}
\email{psfreitas@fc.ul.pt}

\address{Institute of Analysis, University of Ulm, Helmholtzstr.~18, D-89069 Ulm, Germany {\rm{and}}
Institute of Analysis, Dynamics and Modelling, University of Stuttgart, Pfaffenwaldring 57, D-70569 Stuttgart, Germany}
\email{james.kennedy@mathematik.uni-stuttgart.de}

%\date{\today}


\begin{abstract}
We derive inequalities for sums of eigenvalues of Schr\"{o}dinger operators on finite intervals and tori. In 
the first of these cases, the inequalities converge to the classical trace formulae in the limit as the number
of eigenvalues considered approaches infinity.
\end{abstract}


\maketitle

\section{Introduction}

One aspect of the classical theory of Sturm-Liouville problems is concerned with the existence of trace formulae relating,
for instance, a regularized (infinite) sum of eigenvalues to the potential. Given the problem
\begin{equation}
\label{stlibdd}
\left\{
\begin{array}{l}
-u''(x) + q(x) u(x) = \lambda u(x)\eqskip
u(0) =  u(\pi) =  0
\end{array},
\right.
\end{equation}
the classical example due to Gelfand and Levitan in $1953$~\cite{gele} reads as follows
\begin{equation}
\label{classtr}
\sum_{k=1}^{\infty}\left[\lambda_{k}-k^2-\frac{\ds 1}{\ds \pi}\dint_{0}^{\pi} q(t)\dx{t}\right]=
\frac{\ds 1}{\ds 2\pi}\dint_{0}^{\pi} q(t)\dx{t} - \frac{\ds q(0) + q(\pi)}{\ds 4},
\end{equation}
where $\lambda_{k} = \lambda_{k}(q)$ denotes the $k^{\rm th}$ eigenvalue corresponding to the potential $q$.
From the Weyl asymptotics for the eigenvalues of~(\ref{stlibdd}), namely,
\begin{equation}
\label{asymp}
\lambda_{k} = k^2 + \frac{\ds 1}{\ds \pi}\dint_{0}^{\pi} q(t)\dx{t} + {\rm{O}}(k^{-2}),
\end{equation}
we have the convergence of the series on the left-hand side of~\eqref{classtr}.

This and other examples of this type of identity may be found in the books by Levitan and Sargsjan~\cite{lesa,lesa2} (see also~\cite{hakr}) 
and have more recently been shown to have extensions to the case of the perturbed harmonic oscillator on the whole real line -- see~\cite{puso} 
and the references therein.

The main purpose of this paper is to show that these trace identities are limiting situations of inequalities satisfied by finite sums 
of eigenvalues. As we shall see below, in the case of equation~(\ref{stlibdd}) and the corresponding trace formula~(\ref{classtr}), 
if we expand the potential $q(x)$ in a Fourier series
\begin{equation}
\label{fourier}
q(x) = \frac{\ds q_{0}}{\ds 2} + \sum_{k=1}^{\infty} q_{k} \cos(k x),
\end{equation}
where
\begin{equation}
\label{eq:fourier-int}
q_{k} = \frac{\ds 2}{\ds \pi}\dint_{0}^{\pi}q(x) \cos(kx) \dx{x}, \;\; k=0,1,\dots,
\end{equation}
we actually have
\begin{equation}
\label{eq:ex}
\sum_{k=1}^{n}\left[\lambda_{k}-k^2-\frac{\ds 1}{\ds \pi}\dint_{0}^{\pi} q(t)\dx{t}\right]\leq
-\frac{\ds 1}{\ds 2}\sum_{k=1}^{n} q_{2k},
\end{equation}
(see Theorem~\ref{th:finitetracedir}) and it is not very difficult to check that the right-hand side converges to the right-hand
side of~(\ref{classtr}). We also show that there exists no lower bound for the left-hand side of~\eqref{eq:ex}
depending on a finite number of Fourier coefficients alone -- see Remark~\ref{nolowerbound}. It is, however, possible to use our results in combination with an argument given in~\cite{dikii} to obtain an alternative proof of~\eqref{classtr} and we explore this approach in Section~\ref{sec:trace}.

Such inequalities, including for other problems such as the case of Neumann boundary conditions and $N$-dimensional
flat tori, in fact follow from elementary test function arguments. The general principle is that if the potential $q$ is written out
as a Fourier series as in \eqref{fourier}, the eigenfunctions of the corresponding zero-potential problem (in one dimension, sines and
cosines) behave well when tested against the resulting series. The resulting sharp inequalities in terms of the Fourier coefficients
of $q$, which are completely explicit if $q$ is known, also contain interesting special cases under various less explicit assumptions
on $q$ (cf.~Theorems~\ref{th:app} and \ref{torus}).
We remark that the case of (non-periodic) Schr\"odinger operators on the whole line may be addressed via similar methods and this
will appear in another paper~\cite{frke}.

Although the underlying principle is quite simple, surprisingly, its application in this context appears to be new and we have not been able 
to find any similar results in the literature. The earliest results regarding \emph{finite} sums of eigenvalues seem to be those for sums of 
reciprocals of eigenvalues of the Laplacian, typically on inhomogeneous membranes, such as the classical work of P\'olya and Schiffer in 
$1954$~\cite[Chap.~III]{posc}; see also, for example, \cite{laug,lapu}. In higher dimensions, there are bounds for sums of eigenvalues and
spectral zeta functions (cf.~\eqref{speczeta}) of the Laplacian based on the geometry of the underlying domain, such as in \cite{dittm,lasi1,lasi}.

For Schr\"odinger operators, besides the vast literature on Lieb--Thirring inequalities (which are concerned with rather different issues
and bounds from \eqref{eq:ex}; see, e.g., \cite{lawe}), there are well-established bounds on the number of eigenvalues less that a given 
positive constant as in \cite{liya}, which also cover the case of Schr\"odinger operators on domains; see also \cite{flm} and the references 
therein. However, these are all of a fundamentally different nature from \eqref{eq:ex}, typically involving either estimates purely in terms 
of eigenvalues, eigenfunctions and/or geometric or dimensional quantities, or else integral 
expressions for $q$ such as of the form $\int_\Omega (q+\alpha)^{N/2}_- \dx{x}$ for an arbitrary given constant $\alpha \geq 0$ and 
dimension $N \geq 3$, as in \cite{liya}.

This paper is organized as follows. Section~\ref{sec:prelim} lays out the general setting and notation. In Section~\ref{sec:finite}, we 
consider summation bounds of the form of~\eqref{eq:ex} for eigenvalues of Schr\"odinger operators on finite intervals subject to either 
Dirichlet or Neumann boundary conditions (see Theorems ~\ref{th:finitetracedir} and~\ref{th:finitetraceneu}, respectively).  We also sum
the one-dimensional Dirichlet and Neumann bounds in Theorem~\ref{finitetracecomb}, obtaining a bound 
independent of the Fourier coefficients of $q$, and a generalization to ``zeta function''-type bounds, where powers of eigenvalues are
considered (see Theorem~\ref{th:sumpower}).

We then apply Theorems ~\ref{th:finitetracedir} and~\ref{th:finitetraceneu} to a particular class of potentials in Section~\ref{sec:app} (see Theorem~\ref{th:app}),
obtaining a simplified eigenvalue bound in dimension one under additional assumptions on the potential $q$, which include the case where 
$q$ is convex. We show that one can essentially trivially obtain a similar bound on flat tori
in Section~\ref{sec:torus}.

In Section~\ref{sec:trace} we consider the relationship between our finite bounds and (known and potential) trace formulae. We show in particular how our one-dimensional results (Theorems ~\ref{th:finitetracedir} and~\ref{th:finitetraceneu}) 
 can be used to replace some of the arguments of Diki\u{\i}'s proof \cite{dikii} of the Gelfand--Levitan formula \eqref{classtr}, which are arguably more natural than their counterparts in \cite{dikii}.

The final Sections~\ref{sec:zeta} and~\ref{sec:equality} are appendiceal, the former giving generic result which allows our bounds to be generalized 
to powers of eigenvalues (i.e.~zeta functions), and the latter proving the sharpness of (the finite versions of) our inequalities: roughly speaking, 
equality can be achieved in the inequalities for finite sums only if the potential is constant.

\section{Notation and preliminaries\label{sec:prelim}}

We will consider the general Schr\"odinger eigenvalue equation
\begin{equation}
\label{lapln}
	-\Delta u(x) + q(x) u(x) = \lambda u(x),
\end{equation}
where $\Delta = \sum_{i=1}^N \frac{\partial^2}{\partial x_i^2}$ is the Laplace operator on either a finite interval ($N=1$) or an $N$-dimensional torus.
In either case we understand the problem in the usual weak sense. If $q\equiv 0$, \eqref{lapln} reduces to the usual 
Laplacian eigenvalue problem, i.e.~the Helmholtz equation. For general $q \not\equiv 0$, unless otherwise specified we 
make the standing assumption throughout the paper that $q$ may be expanded as an absolutely convergent Fourier series in terms of the eigenfunction 
of the Helmholtz equation (cf.~\eqref{fourier}).

Under these assumptions the operator associated with \eqref{lapln} admits a discrete sequence of eigenvalues
$\lambda_1(q) < \lambda_2(q) \leq \ldots \to \infty$, which we repeat according to multiplicities. We will often abbreviate $\lambda_n(q)$ as 
$\lambda_n$ if there is no danger of confusion, and we will write $\mu_n(q)$ (or $\mu_n$) instead of $\lambda_n(q)$ for the Neumann eigenvalues. 
An (arbitrary) eigenfunction associated with $\lambda_n(q)$ or $\mu_n(q)$ will be denoted by $\varphi_n$. For the case of a zero potential, 
that is, the ordinary Laplacian or Helmholtz 
equation, we will in general write $\lambda_n(0)$ (or $\mu_n(0)$ as appropriate) for the $n^{\rm th}$ ordered eigenvalue and $\psi_n$ for any 
corresponding eigenfunction.

For either a bounded interval or a torus, which we for now denote generically by $\Omega$, given a potential $q:\Omega\to\R$ and a test function $\phi \in V = H^1_0(\Omega)$ or $H^1(\Omega)$ as appropriate, we denote by
\begin{equation}
\label{rayleighsum}
	\mathcal{R}[q,\phi]:= \frac{Q(\phi)}{\|\phi\|_2^2}:= \frac{\dint_\Omega |\nabla\phi(x)|^2\dx{x}+\dint_\Omega q(x)\phi^2(x)\dx{x}}
{\dint_\Omega \phi^2(x)\dx{x}}
\end{equation}
the Rayleigh quotient associated with $q$ at $\phi$, where $Q(\phi)=Q(\phi,\phi)$ is the bilinear form associated with the Schr\"odinger operator. 
A standard generalization of the usual min-max formula for the eigenvalues 
of \eqref{lapln} states that if $\phi_1,\ldots,\phi_n$ is a collection of $n \geq 1$ such test functions mutually orthogonal in $L^2(\Omega)$, then
\begin{equation}
\label{sumprinciple}
	\sum_{k=1}^{n} \lambda_k(q) \leq \sum_{k=1}^{n} \mathcal{R}[q,\phi_k]
\end{equation}
(see, e.g., \cite{band}). If $q\equiv 0$ and we consider \eqref{lapln} with Dirichlet boundary conditions on $(0,\pi)$, then $\lambda_n(0) = n^2$ 
with associated eigenfunction $\psi_n(x) = \sin(nx)$, meaning in particular that if $s>1/2$, then the infinite sum of powers of eigenvalues
\begin{displaymath}
	\sum_{k=1}^\infty \lambda_k^{-s}(0) = \sum_{k=1}^\infty \frac{1}{k^{2s}} = \zeta(2s)
\end{displaymath}
equals the Riemann zeta function, which we write as $\zeta(2s) =: \zeta_0(s)$. By way of analogy we define, as standard, the generalized 
(or spectral) zeta function associated with the potential $q$ as
\begin{equation}
\label{speczeta}
	\zeta_q(s) := \sum_{k=1}^\infty \lambda_k^{-s}(q).
\end{equation}
Finally, we denote by $\langle \,.\,,\,.\, \rangle$ the usual inner product in $L^2(\Omega)$, and by $\|\,.\,\|_p$ the $L^p$-norm on 
$\Omega$, $1\leq p \leq \infty$.

\section{Summation bounds on finite intervals}
\label{sec:finite}

We first treat one-dimensional Schr\"odinger operators on the interval $(0,\pi)$. This means we can write the potential $q$ as a cosine Fourier series
as in~\eqref{fourier}, where we note in particular that $q_0/2$ is the average value of $q$,
\begin{displaymath}
\frac{q_0}{2} = \dashint_0^\pi q(x)\dx{x} := \frac{1}{\pi}\dint_0^\pi q(x)\dx{x}.
\end{displaymath}
Our starting point is the following bound for the Dirichlet problem.
\begin{theorem}
\label{th:finitetracedir}
If in the problem \eqref{stlibdd} the potential $q$ admits the expansion \eqref{fourier}, then its eigenvalues $\lambda_k = \lambda_k(q)$ satisfy
\begin{equation}
\label{finitetracedir}
\sum_{k=1}^{n}\left(\lambda_{k}-k^2-\frac{\ds q_{0}}{\ds 2}\right)\leq
-\frac{\ds 1}{\ds 2}\sum_{k=1}^{n} q_{2k}.
\end{equation}
for all $n\geq 1$. Equality for any $n\geq 1$ implies $q$ is constant. As $n$ goes to infinity, the right-hand side of the above inequality converges to
\begin{equation}
\label{tracedir}
\frac{1}{2}\,\dashint_0^\pi q(x)\dx{x}- \frac{\ds q(0) + q(\pi)}{\ds 4}.
\end{equation}
\end{theorem}

\begin{proof}
Using $\psi_k(x)=\sin(kx)$, $k=1,\ldots,n$, as test functions in \eqref{rayleighsum} we have
\begin{displaymath}
\begin{split}
\sum_{k=1}^{n}\lambda_{k} & \leq \sum_{k=1}^{n}\frac{\ds \dint_{0}^{\pi} k^2\cos^{2}(kx) +
q(x)\sin^{2}(k x)\dx{x}}{\ds \dint_{0}^{\pi} \sin^{2}(k x)\dx{x}}\\
& = \sum_{k=1}^{n}\left[ k^2 + \frac{\ds 2}{\ds \pi}\dint_{0}^{\pi} q(x)\sin^{2}(k x)\dx{x}\right].
\end{split}
\end{displaymath}
By expanding $q$ as in \eqref{fourier}, and using the identities $\sin^{2}(k x)=(1-\cos(2kx))/2$ and 
$\cos(jx) \cos(2kx) = (\cos(jx+2kx)+\cos(jx-2kx))/2$, we obtain
\begin{displaymath}
\begin{split}
\dint_{0}^{\pi} q(x)\sin^{2}(k x)\dx{x} & = \frac{\ds 1}{\ds 2}\dint_{0}^{\pi} \left[\frac{\ds q_{0}}{\ds 2} +
\sum_{j=1}^{\infty} q_{j} \cos(j x)\right]\left[1-\cos(2kx)\right]\dx{x}\\
& = \frac{\ds \pi}{\ds 4}(q_{0} - q_{2k}).
\end{split}
\end{displaymath}
Replacing this in the expression above yields
\begin{displaymath}
\sum_{k=1}^{n}\lambda_{k} \leq \sum_{k=1}^{n} \left[k^2 + \frac{\ds q_{0} - q_{2k}}{\ds 2}\right]
\end{displaymath}
as desired. If we evaluate the Fourier series for $q$ at $x=0$ and at $x=\pi$ and add the resulting sums we see that
\begin{equation}
\label{potentialvalue}
q(0) + q(\pi) = q_{0} + 2\sum_{k=1}^{\infty}q_{2k},
\end{equation}
showing that the limit of the sum on the right-hand side of inequality \eqref{finitetracedir} is given by~(\ref{tracedir}). The statement in case of equality follows from Theorem~\ref{th:uniqueness}.
\end{proof}

\begin{remark}\label{nolowerbound}
(i) It is not possible to find a \emph{lower} bound for $\sum_{k=1}^n \lambda_k$ depending only on the first $2n$ Fourier coefficients of $q$ (or indeed, the first $m$ for any fixed $m\geq 0$) and the eigenvalues $k^2$, as the following simple example shows. For arbitrary $n\geq 1$, if we let $q(x)=t\cos(2n+2)x$, where $t>0$ is taken very large, then we have $q_{2n+2}=t$, while $q_k=0$ for all other $k\geq 0$. In this case, using an argument as in the proof of Theorem~\ref{th:finitetracedir} with $\sin(x), \ldots,\sin(n-1)x$, $\sin(n+1)x$ as our $n$ test functions, we see
\begin{displaymath}
	\sum_{k=1}^n \lambda_k \leq \sum_{k=1}^{n-1}k^2 + (n+1)^2 - \frac{t}{2} \longrightarrow -\infty
\end{displaymath}
if we let $t \to \infty$, even though $q_0=\ldots=q_{2n}=0$. We could construct a similar example which also satisfies $q(0)=q(\pi)=0$ by taking, 
for example, $q(x)=t\cos(2n+2)x - t\cos(2n+4)x$. It  seems that any lower bound would have to take into account a quantity such as 
$\sup_{n\in\N}|q_{2n}|$ or $\sup_{x\in(0,\pi)}|q(x)|$, or else only be valid asymptotically (e.g.~via the inclusion of an $O(n^s)$-type error 
term). It is easy to construct analogous examples for the other problems we will consider.

(ii) The inequality \eqref{finitetracedir} is valid without convergence of the Fourier series \eqref{fourier}, provided only that the
coefficients $q_k$
given by \eqref{eq:fourier-int} are well defined.
\end{remark}

There is a direct analogue of Theorem~\ref{th:finitetracedir} for the Neumann problem.

\begin{theorem}
\label{th:finitetraceneu}
For all $n\geq 0$,
\begin{equation}
\label{finitetraceneu}
\sum_{k=0}^{n}\left(\mu_{k}-k^2-\frac{\ds q_{0}}{\ds 2}\right)\leq
\frac{\ds 1}{\ds 2}\sum_{k=1}^{n} q_{2k}.
\end{equation}
Equality for some $n\geq 0$ implies that $q$ is constant. The right-hand side of the above inequality converges to
\begin{equation*}
\label{traceneu}
-\frac{1}{2}\,\dashint_0^\pi q(x)\dx{x} + \frac{\ds q(0) + q(\pi)}{\ds 4}
\end{equation*}
as $n$ goes to infinity.
\end{theorem}

\begin{proof}
The proof is exactly the same as for Theorem~\ref{th:finitetracedir}, except that we use $\psi_k(x)=\cos(kx)$, $k=0,1,\ldots,n$ as test 
functions. We omit the details.
\end{proof}

As in the Dirichlet case, a classical trace formula analogous to \eqref{classtr} (see e.g. \cite[Sec.~1.14]{lesa2}) implies there is equality in the
limit in \eqref{finitetraceneu} as $n \to \infty$.

By combining our separate estimates for Dirichlet and Neumann eigenvalues, we can simplify the resulting sums. This effectively corresponds 
to considering the eigenvalues of the circle; cf.~Theorem~\ref{torus}. In this case, however, the result is an immediate consequence of
Theorems~\ref{th:finitetracedir} and \ref{finitetraceneu}.

\begin{theorem}
\label{finitetracecomb}
For all $n \geq 1$,
\begin{equation}
\label{eq:finitetracecomb}
	\sum_{k=1}^n \left[\frac{\lambda_k+\mu_k}{2} - k^2 - \dashint_0^\pi q(x)\dx{x}\right]
	+\frac{1}{2}\left[\mu_0-\dashint_0^\pi q(x)\dx{x}\right] \leq 0.
\end{equation}
with equality for any $n\geq 1$ implying $q$ is constant.
\end{theorem}

\begin{remark}
\label{rem:finitetracecomb}
As in the Dirichlet and Neumann cases, the Weyl asymptotics imply that the left-hand side of \eqref{eq:finitetracecomb} converges as $n \to \infty$,
and by combining the separate trace formulae for $\lambda_k$ and $\mu_k$ we see there is again equality in the limit.
\end{remark}

This combination of boundary conditions also allows us to obtain the following ``zeta function"-type bound.

\begin{theorem}
\label{th:sumpower}
Suppose that $\mu_0>0$. Then for all $n\geq 1$ and $s>0$,
\begin{displaymath}
	\sum_{k=1}^n \lambda_k^{-s}+\sum_{k=0}^n \mu_k^{-s} \geq 2\sum_{k=1}^n 
	\left[k^2 + \dashint_0^\pi q(x)\dx{x}\right]^{-s} + \left[\dashint_0^\pi q(x)\dx{x}\right]^{-s}.
\end{displaymath}
\end{theorem}

If $s>1/2$, then both sides of the above inequality converge as $n \to \infty$. The proof 
is a variant of that of Theorem~\ref{th:zeta} and is therefore delayed until Section~\ref{sec:zeta}.

\section{An application to a particular class of potentials on the interval}
\label{sec:app}

Here we give an application, or special case, of Theorems~\ref{th:finitetracedir} and \ref{th:finitetraceneu}. As in Section~\ref{sec:finite}, we denote by
$\lambda_k$ and $\mu_k$ the ordered Dirichlet and Neumann eigenvalues associated with $q$, respectively.

\begin{theorem}
\label{th:app}
Suppose that $q$ admits the expansion \eqref{fourier} and is absolutely continuous on $(0,\pi)$.
\begin{itemize}
\item[(i)] If $q'(x) \leq q'(\pi-x)$ a.e. on $(0,\frac{\pi}{2})$, then for all $n \geq 1$,
\begin{displaymath}
	\sum_{k=1}^n \left(\lambda_k- k^2 - \dashint_0^\pi q(x) \dx{x}\right) \leq 0.
\end{displaymath}
\item[(ii)] If $q'(x) \geq q'(\pi-x)$ a.e. on $(0,\frac{\pi}{2})$, then for all $n \geq 0$,
\begin{displaymath}
	\sum_{k=0}^n \left(\mu_k - k^2 - \dashint_0^\pi q(x) \dx{x}\right) \leq 0.
\end{displaymath}
\item[(iii)] Under the assumptions of (i), if in addition $\int_0^\pi q(x) \dx{x} \leq 0$, then for all $s>1/2$ we also have
\begin{displaymath}
	\sum_{k=1}^n \lambda_k^{-s} \geq \sum_{k=1}^n k^{-2s}
\end{displaymath}
for all $n \geq 1$, and $\zeta_q(s) \geq \zeta(2s)$.
\end{itemize}
Equality in any of the above finite inequalities implies that $q$ is constant on $(0,\pi)$.
\end{theorem}

\begin{remark}
The assumptions of the Dirichlet case (i) are always satisfied by convex potentials, i.e.~potentials $q$ for which $q''(x) \geq 0$ a.e.~in $(0,\pi)$, while concave potentials, i.e.~with $q''(x) \leq 0$ a.e., always satisfy the Neumann condition~(ii).
\end{remark}

\begin{proof}[Proof of Theorem~\ref{th:app}]
(i) Fix $n \geq 1$. Recalling that $\dashint_0^\pi q(x)\dx{x} = q_0/2$ and using Theorem~\ref{th:finitetracedir}, we only need to show that 
the right-hand side of \eqref{finitetracedir} is non-positive. Recalling the definition of $q_{2k}$ and integrating by parts,
\begin{displaymath}
\begin{split}
	q_{2k} = \frac{2}{\pi} \int_0^\pi q(x) \cos(2kx)\dx{x} &= \frac{2}{\pi} \int_0^\frac{\pi}{2} 
	\left[q(x)+q(\pi-x)\right] \cos(2kx)\dx{x}\\
	& = -\frac{1}{\pi k}\int_0^\frac{\pi}{2} \left[q'(x)-q'(\pi-x)\right] \sin(2kx)\dx{x}.
\end{split}
\end{displaymath}
Summing over $k$ and noting that $\frac{d}{dx} \cos^2(kx) /k = -\sin(2kx)$, this means
\begin{displaymath}
	-\frac{1}{2}\sum_{k=1}^n q_{2k} = -\frac{1}{2\pi}\int_0^\frac{\pi}{2} \left[q'(x)-q'(\pi-x)\right] \frac{d}{dx}\left(
	\sum_{k=1}^n \frac{\cos^2(kx)}{k^2}\right)\dx{x}.
\end{displaymath}
It is known that $\sum_{k=1}^n \cos^2(kx)/k^2$ is decreasing on $(0,\pi/2)$ for every $n \geq 1$ (cf. \cite[pp.~322--3]{laug}). Our assumptions on
$q$ 
therefore imply that the above integrand is positive for almost all $x \in (0,\pi/2)$, and thus
\begin{displaymath}
	\sum_{k=1}^n \left(\lambda_k - k^2 - \frac{q_0}{2}\right) \leq -\frac{1}{2} \sum_{k=1}^n q_{2k} \leq 0.
\end{displaymath}
Equality for some $n \geq 1$ means that
\begin{displaymath}
	0 = \sum_{k=1}^n \left(\lambda_k - k^2 -\frac{q_0}{2}\right) \leq - \frac{1}{2} \sum_{k=1}^n q_{2k} \leq 0,
\end{displaymath}
so that every inequality is an equality. In this case there is also equality in Theorem~\ref{th:finitetracedir}, and so $q$ is constant.

(ii) Applying Theorem~\ref{th:finitetraceneu} in place of Theorem~\ref{th:finitetracedir}, we obtain
\begin{displaymath}
\begin{split}
	\sum_{k=0}^n \left(\mu_k - k^2 - \frac{q_0}{2}\right) &\leq \frac{1}{2}\sum_{k=1}^n q_{2k} \\
	&\leq\frac{1}{2\pi}\int_0^\frac{\pi}{2} \left[q'(x)+q'(\pi-x)\right]\frac{d}{dx}\left(\sum_{k=1}^n\frac{\cos^2(kx)}{k^2}\right)\dx{x}.
\end{split}
\end{displaymath}
This time the integrand is negative almost everywhere. The case of equality follows in the same way as in (i).

(iii) This follows from Theorem~\ref{th:zeta} with $a_k = b_k = k^2$.
\end{proof}

\section{Summation bounds for the circle and flat torus}
\label{sec:torus}

In the case of $N$-dimensional flat tori we can obtain an especially simple bound which may be seen as the natural multi-dimensional generalization of Theorem~\ref{finitetracecomb}. By writing the zero-potential eigenfunctions as complex exponentials instead of sines and cosines, as is more natural on a torus, the proof becomes essentially trivial.

In order to proceed, we shall need some notation. We will denote by $\mathbb{T}$ an $N$-dimensional flat torus, $N\geq 1$, spanned by linearly independent vectors $v_1,\ldots,v_N \in \R^N$. 
That is, if we define a lattice $\Gamma \subset \R^N$ as
\begin{displaymath}
\Gamma = \{n_1 v_1+\ldots+n_N v_N: n_i \in \mathbb{Z}, \,i=1\,\ldots,N\}.
\end{displaymath}
and an action of $\Gamma$ on $\R^N$ by $\gamma(x):=\gamma+x$, $\gamma\in\Gamma$, $x\in\R^N$, then our torus is given by 
$\mathbb{T} = \R^N / \Gamma$. If $N=1$ then of course $\mathbb{T}$ is the circle.

We define the vectors $w_1,\ldots,w_N \in \R^N$ by $(w_j,v_k) = \delta_{jk}$, the Kronecker delta, where $( \,.\, , \,.\, )$ is the usual inner 
product on $\R^N$. Denote by $W$ the matrix whose $j^{\rm th}$ row is given by the vector $w_j$. Then for each $\alpha =
(\alpha_1,\ldots,\alpha_N) \in \Z^N$, we may define an eigenfunction $\psi_\alpha$ of the zero-potential problem \eqref{lapln} with $q=0$ on the
manifold without boundary $\mathbb{T}$ by
\begin{displaymath}
	\psi_\alpha(x):= e^{2\pi i \alpha^{\textrm T} W x};
\end{displaymath}
if we denote the $(m,n)^{\rm th}$ entry of $W$ by $w_{mn}$, then the associated eigenvalue $\lambda_\alpha = \lambda_\alpha(0)$  is given by
\begin{displaymath}
	\lambda_\alpha = 4\pi^2 \sum_{n=1}^N \left(\sum_{m=1}^N \alpha_m w_{mn}\right)^2.
\end{displaymath}
In the one-dimensional case, we have just one vector $v=2\pi a \in \R$ (without loss of generality $v\geq 0$), in this case just one vector 
$w =v^{-1} > 0$, and the eigenfunctions become $4\psi_n = e^{inx/a}$, with $\lambda_n = (n/a)^2$, $n \in \Z$.

Ordering the eigenvalues $\lambda_\alpha$ as an increasing sequence $\{\lambda_k(0)\}_{k\in\Z}$ with $\lambda_0(0)=0$ (corresponding to $\alpha = 0$), for each $k\geq 1$ there exists $\alpha=\alpha(k)$ for which $\lambda_k(0) = \lambda_\alpha$. The exact relationship between $\alpha$ and $k$ depends on the $v_j$, and to the best of our knowledge there is no known explicit formula for this for arbitrary $v_j$. We will order the eigenvalues $\lambda_k(q)$ of \eqref{lapln} with potential $q$ on $\mathbb{T}$ similarly.

\begin{theorem}
\label{torus}
For any integrable $q$ and for all $n \geq 0$,
\begin{equation}
\label{finitetracetor}
	\sum_{k=0}^n \left[\lambda_k(q)-\lambda_k\left(\dashint_{\mathbb{T}}q(x)\dx{x}\right) \right]\leq 0.
\end{equation}
\end{theorem}

\begin{proof}
We use the $\psi_\alpha$ as test functions in the Rayleigh quotient \eqref{rayleighsum}: for any $\alpha \in \Z^N$, we have
\begin{displaymath}
	\mathcal{R}[q,\psi_\alpha] = \lambda_\alpha + \frac{\dint_{\mathbb{T}} q(x) |\psi_\alpha(x)|^2\dx{x}}{\dint_{\mathbb{T}} |\psi_\alpha(x)|^2\dx{x}} =\lambda_\alpha+\dashint_{\mathbb{T}}q(x) \dx{x},
\end{displaymath}
since obviously $|\psi_\alpha(x)|=1$ for all $x \in \mathbb{T}$ and all $\alpha\in \Z^N$. Choosing the first $n$ of these functions and using the principle \eqref{sumprinciple} gives us the inequality.
\end{proof}

\begin{theorem}
\label{th:toruspower}
Suppose that $\lambda_0(q)>0$. Then for all $n\geq 1$ and $s>0$,
\begin{displaymath}
	\sum_{k=0}^n \lambda_k^{-s}(q) \geq\sum_{k=0}^n\left[\lambda_k\left(
	\dashint_{\mathbb{T}}q(x)\dx{x}\right)\right]^{-s}.
\end{displaymath}
\end{theorem}

\begin{proof}
Since $0<\lambda_0(q)\leq \dashint_{\mathbb{T}}q(x)\dx{x}$ by Theorem~\ref{torus}, we may apply Theorem~\ref{th:zeta}, from which the conclusion 
follows immediately.
\end{proof}

\section{On the associated trace formulae}
\label{sec:trace}

We have already observed that there is equality in the limit as $n \to \infty$ in Theorems~\ref{th:finitetracedir},~\ref{th:finitetraceneu} and~\ref{finitetracecomb}, since we have convergence to the classical trace formulae of Gelfand--Levitan type. What is interesting is that the theorems from Section~\ref{sec:finite} allow us to
obtain a new (part of a) proof of the trace formulae. Our starting point is a paper by Diki\u{\i}~\cite{dikii}, who gave an alternative proof of the trace formula which is, in some sense, more natural that that in~\cite{gele}, which was based on a study of the asymptotics of the associated Green's functions.

This proof involves a two-part argument, which in our notation is as follows. If we denote the ordered eigenfunctions associated with the zero potential by $\psi_k$ and those associated with $q$ by $\varphi_k$, and assuming without loss of generality that the mean value $q_0$ of $q$ is zero, Diki\u{\i} proved using trigonometric identities, integration by parts and a manipulation of the resulting sums that
\begin{equation}
\label{eq:dikii1}
\sum_{k=1}^n \left(\mathcal{R}[q,\psi_k] - k^2\right) \longrightarrow -\frac{q(0)+q(4)}{2}
\end{equation}
as $n \to \infty$. The second part of the  proof consists in using eigenvalue and eigenfunction asymptotics to show that
\begin{equation}
\label{eq:dikii2}
\lim_{n\to\infty} \sum_{k=1}^n \left(\mathcal{R}[q,\varphi_k]-\mathcal{R}[q,\psi_k]\right)=0.
\end{equation}
If we denote by $H = -\Delta +q$ the Schr\"odinger operator on $L^2(0,\pi)$ associated with $q$ (and the Dirichlet boundary condition), then we may 
formally rewrite \eqref{eq:dikii2} as
\begin{equation}
\label{eq:basisform}
\lim_{n \to \infty} \sum_{k=1}^n \left(\langle \varphi_k,H\varphi_k\rangle - \langle \psi_k,H\psi_k\rangle\right) = 0.
\end{equation}
In some sense we can think of \eqref{eq:basisform} as asserting that the 
action of $H$ is invariant with respect to a ``change of basis'' from $\{\varphi_k\}$ to $\{\psi_k\}$; we note that the arguments in \cite{dikii} were
phrased in these terms, and did not involve the use of Rayleigh quotients. Indeed, for $k,m \geq 1$,
we shall write $a^k_m = \langle \psi_k, \varphi_m \rangle$, so that $\psi_k = \sum_{m=1}^\infty a^k_m \varphi_m$ and $\varphi_m = \sum_{k=1}^\infty a^k_m \psi_k$, with $\sum_{k=1}^\infty (a^k_m)^2 = \sum_{m=1}^\infty (a^k_m)^2 = 1$ under the normalization
$\|\psi_k\|_2 = \|\varphi_m\|_2 = 1$. Then the sum \eqref{eq:dikii2} is equivalent to
\begin{equation}
\label{eq:dikiiestsum}
\begin{split}
\sum_{k=1}^\infty \sum_{m=1}^\infty & \lambda_k(a^k_m)^2 - \sum_{k=1}^n \sum_{m=1}^\infty \lambda_m (a^k_m)^2
=\\ &\sum_{k=1}^n \sum_{m=n+1}^\infty \lambda_k \left((a^m_k)^2 - (a^k_m)^2\right)
+ \sum_{k=1}^n \sum_{m=n+1}^\infty (\lambda_k - \lambda_m)(a^k_m)^2.
\end{split}
\end{equation}
Diki\u{\i} showed that for the problem on the interval the two sums on the right-hand side of \eqref{eq:dikiiestsum} tend to zero as $n \to \infty$ by using the basic asymptotic estimates $\lambda_k = k^2 + O(1)$ and $\varphi_k = \psi_k + O(k^{-1})$, the latter holding uniformly in $x \in (0,\pi)$ (see, e.g., \cite{lesa2}).

The proof of Theorem~\ref{th:finitetracedir}, whose conclusion says exactly that $\sum_{k=1}^n (\mathcal{R}[q,\varphi_k]-\mathcal{R}[q,\psi_k]) \leq 0$ 
for all $n \geq 1$, also gives an alternative proof of \eqref{eq:dikii1} along the way.
In a sense this is 
more natural than the proof in~\cite{dikii}, at least when taken together with \eqref{eq:dikii2}, since it computes the finite sum
$\sum_{k=1}^n \mathcal{R} [q,\psi_k]$ explicitly in terms of $q$.

\begin{remark}
It would be interesting to know if such arguments might also work in higher dimensions, where there are next to no known trace formulae. We will
not explore this here, but as an example we remark that equality in the limit as $n \to \infty$ in Theorem~\ref{torus} is equivalent
to~\eqref{eq:dikii2} holding on the torus. The difficulty in using Diki\u{\i}'s idea in higher dimensions is that the asymptotic behaviour of
the $\lambda_k$ and $\varphi_k$ changes; for example, in two dimensions, we now have $\lambda_k \sim k$, not $k^2$. This makes it harder to
obtain effective bounds on the sums in \eqref{eq:dikiiestsum}.
\end{remark}

\section{Generalization to power bounds and zeta functions}
\label{sec:zeta}

Here we will prove a theorem from which the power bound generalizations stated in earlier sections, such as Theorem~\ref{th:app} (iii), will follow
immediately; it was inspired by, and based upon, a similar result and argument in \cite[Sec.~4]{laug}. We also give the very similar proof of Theorem~\ref{th:sumpower}.

\begin{theorem}
\label{th:zeta}
Suppose the sequences $(\lambda_k)_{k\geq 1}$ and $(a_k)_{k\geq 1}$ are positive and that $(a_k)_{k\geq 1}$ is non-decreasing in $k\geq 1$. Suppose also that the sequence $(b_k)_{k\geq 1}$ satisfies $\sum_{k=1}^m \lambda_k \leq \sum_{k=1}^m b_k$ for all $m\geq 1$. Then for all $s>0$ and all $n\geq 1$ we have
\begin{equation}
\label{zeta}
\sum_{k=1}^n (\lambda_k)^{-s} \geq \sum_{k=1}^n \left( (s+1)(a_k)^{-s}-s(a_k)^{-s-1}b_k\right).
\end{equation}
If the sequence $(b_k)_{k\geq 1}$ is itself positive and non-decreasing in $k\geq 1$, then the right-hand side of \eqref{zeta} is maximized when $a_k=b_k$ for all $1\leq k\leq n$.
\end{theorem}

Notice that if $a_k=b_k$ for all $k\geq 1$, then \eqref{zeta} simplifies to $\sum_{k=1}^n (\lambda_k)^{-s} \geq \sum_{k=1}^n (b_k)^{-s}$.

\begin{proof}
We will use the notation $[y]_+$, $y\in\R$, to mean the expression taking on the value $y$ if $y\geq 0$ and zero otherwise; $[f(x)]_{g(x)\geq y}$
will 
represent $f(x)$ if $g(x) \geq y$ and zero otherwise. We start with the following identity, valid for $\lambda > 0$,
\begin{equation}
\label{powerrep}
	\lambda^{-s} = s(s+1)\dint_0^\infty \alpha^{-s-2}[\alpha-\lambda]_+\dx{\alpha}
\end{equation}
For $n\geq 1$, $s>0$ arbitrary, applying this to both $\lambda_k$ and $a_k$, we have
\begin{displaymath}
\begin{split}
\sum_{k=1}^n \left(\lambda_k^{-s}-a_k^{-s}\right) &= s(s+1)\dint_0^\infty \alpha^{-s-2}
	\sum_{k=1}^n \left([\alpha-\lambda_k]_+-[\alpha-a_k]_+\right)\dx{\alpha}\\
	&\geq s(s+1)\dint_0^\infty \alpha^{-s-2}\sum_{k=1}^n [a_k-\lambda_k]_{\alpha\geq a_k}\dx{\alpha}\\
	&\geq s(s+1)\dint_0^\infty \alpha^{-s-2}\sum_{k=1}^n [a_k-b_k]_{\alpha\geq a_k}\dx{\alpha}\\
	&=\sum_{k=1}^n s(s+1)(a_k-b_k)\dint_{a_k}^\infty \alpha^{-s-2}\dx{\alpha},
\end{split}
\end{displaymath}
which after simplification and rearrangement gives us \eqref{zeta}. (Note that we needed the sequence $a_k$ to be weakly increasing to justify the third line above.) For the maximizing property we consider each term on the right-hand side of \eqref{zeta} as a function of $a_k$ by setting $g_k(a_k):= (s+1)(a_k)^{-s}-s(a_k)^{-s-1}b_k$. Differentiating in $a_k$ shows that $g_k$ reaches its unique maximum when $a_k=b_k$.
\end{proof}

\begin{proof}[Proof of Theorem~\ref{th:sumpower}]
Keeping the notation from the proof of Theorem~\ref{th:zeta} and using \eqref{powerrep}, we have
\begin{displaymath}
\begin{split}
\sum_{k=1}^n \lambda_k^{-s} &+\sum_{k=0}^n \mu_k^{-s}
=s(s+1)\dint_0^\infty \alpha^{-s-2}\left(\sum_{k=1}^n [\alpha-\lambda_k]_+ +\sum_{k=0}^n
	[\alpha-\mu_k]_+\right)\dx{\alpha}\\
&\geq s(s+1)\dint_0^\infty \alpha^{-s-2}\left(\sum_{k=1}^n [2\alpha-\lambda_k-\mu_k]_{\alpha \geq k^2+\frac{q_0}{2}}+
	[\alpha-\mu_0]_{\alpha\geq \frac{q_0}{2}}\right)\dx{\alpha}.\\
\end{split}
\end{displaymath}
For each fixed $\alpha\geq q_0/2+1$, the sum in the latter integral is from $k=1$ to some $m=m(\alpha) \leq n$; if $\alpha \in [q_0/2, q_0/2+1)$, then the bracketed term reduces to $\alpha-\mu_0$, and otherwise it is zero. This means that for each fixed $\alpha\geq q_0/2+1$ we may apply Theorem~\ref{finitetracecomb} (or Theorem~\ref{th:finitetraceneu} with $n=0$ if $\alpha \in [q_0/2, q_0/2+1)$) to obtain
\begin{displaymath}
\sum_{k=1}^{m(\alpha)} [2\alpha-\lambda_k-\mu_k]_{\alpha \geq k^2+\frac{q_0}{2}}
+[\alpha-\mu_0]_{\alpha\geq \frac{q_0}{2}}\leq \sum_{k=1}^{m(\alpha)} [2\alpha-2k^2-q_0]_+ +[\alpha-q_0/2]_+
\end{displaymath}
for each $\alpha>0$. Substituting this back into the above expression for $\sum_{k=1}^n \lambda_k^{-s} +\sum_{k=0}^n \mu_k^{-s}$ and 
applying \eqref{powerrep} in the other direction yields the theorem.
\end{proof}

\section{The case of equality}
\label{sec:equality}

Finally, we will prove the sharpness of our inequalities, in the sense that equality for some $n\geq 1$ in a bound of the form
\eqref{sumprinciple} forces the potential $q$ to be a constant. This stems from the fact that the only functions that can minimize the Rayleigh 
quotient expression \eqref{rayleighsum} are sums of eigenfunctions of the corresponding equation. Although we doubt this is new, we do not know 
of any explicit reference in the literature and so give a proof here. 
We suppose we have the equation \eqref{lapln} in any one of the cases considered and two different potentials $q_1, q_2\in L^\infty(\Omega)$. 
We denote by $V = H^1_0 (\Omega)$ or $H^1(\Omega)$ the appropriate Hilbert space. We will write $\lambda_k(q_2)$, $k\geq 1$ for the eigenvalues of the problem associated with $q_2$, ordered by increasing size and repeated according to multiplicities, and  $\psi_k$ for the corresponding eigenfunctions which form an orthonormal basis of $L^2(\Omega)$.

\begin{lemma}
\label{lemma:equality}
Suppose that for some $n\geq 1$,
\begin{equation}
\label{eq}
\sum_{k=1}^{n} \lambda_k(q_1) = \sum_{k=1}^n \mathcal{R}[q_1,\psi_k].
\end{equation}
Then there exist eigenfunctions $\varphi_k$, $k\geq 1$ corresponding to $\lambda_k(q_1)$ (ordered by increasing magnitude) such that
\begin{equation}
\label{span}
	\spn\{\psi_1,\ldots,\psi_n\} = \spn\{\varphi_1,\ldots,\varphi_n\}
\end{equation}
in $L^2(\Omega)$, that is, each $\psi_k$ may be expressed as a finite linear combination $\psi_k(x)=\sum_{m=1}^n a^k_m\varphi_m(x)$ for
suitable constants $a^k_m \in\R$.
\end{lemma}

Of course, if $\lambda_n(q_1)$ is not simple, then we need to choose the right eigenfunction(s) $\varphi_n$ (and possibly $\varphi_{n-1},\ldots,\varphi_{n-m}$) in the corresponding eigenspace. We also note that this is really an abstract result which is true for any two positive, self-adjoint operators on a (real) Hilbert space, and in particular valid in greater generality.
In general, however, it does not seem so easy to prove that $q_1-q_2$ is constant in $\Omega$ (following from $a^k_m = \delta_{km}$ in Lemma~\ref{lemma:equality}, so that the $\psi_k$ are directly eigenfunctions of $q_1$). Here, we will deal only with the case of an interval with Dirichlet or Neumann boundary conditions, and with $q_2=0$. A key element of our proof, as in Section~\ref{sec:finite}, is the fact that products of eigenfunctions of the zero potential, namely sines and cosines, are mutually orthogonal in the relevant $L^2$-space. We expect the same idea should work on the $N$-dimensional 
torus, since the eigenfunctions are complex exponentials, but the argument is complicated by various issues related to the multiplicity of the eigenvalues, and we do not explore it here. So we now return to having $q_2\equiv 0$ and labelling $q_1$ as $q$, given by \eqref{fourier}.

\begin{theorem}
\label{th:uniqueness}
Suppose that for some $n\geq 1$, there is equality in \eqref{finitetracedir}, respectively \eqref{finitetraceneu}. Then $q(x)$ is constant in $(0,\pi)$ with eigenfunctions given by $\varphi_k(x)=\sin(kx)$, $k\geq 1$, and $\varphi_k(x)=\cos(kx)$, $k\geq 0$, in the Dirichlet and Neumann cases, respectively.
\end{theorem}

\begin{proof}[Proof of Lemma~\ref{lemma:equality}]
For each $k\geq 1$, as in Section~\ref{sec:trace} we write $\psi_k = \sum_{m=1}^\infty a_m^k \varphi_m$, where 
$a_m^k = \langle \psi_k,\varphi_m \rangle$ and the $\varphi_k$ are, for the meantime, an arbitrary set of eigenfunctions for $q_1$,
in the sense that we allow an arbitrary decomposition of any eigenspace of dimension $\geq 2$. 
The orthonormality relations imply $\sum_{m=1}^\infty a_m^k a_m^l = \delta_{kl}$. Using \eqref{eq} and denoting by $Q_1$ the bilinear 
form associated with $q_1$ (cf~\eqref{rayleighsum}) we have
\begin{displaymath}
\sum_{k=1}^n \lambda_k(q_1) = \sum_{k=1}^n \mathcal{R}[q_1,\psi_k]
= \sum_{k=1}^n Q_1 \left( \sum_{m=1}^\infty a_m^k \varphi_m, \sum_{l=1}^\infty a_l^k \varphi_l \right).
\end{displaymath}
Since everything converges, and since $Q_1(\varphi_m,\varphi_l) = \delta_{ml} \lambda_m (q_1)$, this reduces to
\begin{equation}
\label{weightesums}
	\sum_{k=1}^n \lambda_k(q_1) = \sum_{k=1}^n \sum_{m=1}^\infty (a_m^k)^2 \lambda_m (q_1) 
	= \sum_{m=1}^\infty \left( \sum_{k=1}^n (a_m^k)^2\right) \lambda_m (q_1).
\end{equation}
Since the functions $\varphi_m = \sum_{k=1}^\infty a_m^k \psi_k$ are also normalized, we have $\sum_{k=1}^n (a_m^k)^2 \leq 1$ for all $m\geq 1$ 
and $\sum_{m=1}^\infty\left(\sum_{k=1}^n (a_m^k)^2\right)=n$. Hence the only way we can have equality in \eqref{weightesums} 
is if the coefficient of $\lambda_m (q_1)$ in the sum on the right-hand side is zero whenever $\lambda_m (q_1) > \lambda_n (q_1)$, which 
means by definition of the $a_m^k$ that $\spn_{L^2(\Omega)}\{\psi_k\}_{k=1}^n$ is contained in the union of the eigenspaces associated with 
$\lambda_1(q_1),\ldots,\lambda_n(q_1)$ (which may be more than $n$-dimensional if $\lambda_n$ is not simple). However, since 
$\spn_{L^2(\Omega)}\{\psi_k\}_{k=1}^n$ is $n$-dimensional, we can find a decomposition of the eigenspace associated with $\lambda_n(q_1)$ 
so that $\spn_{L^2(\Omega)}\{\psi_k\}_{k=1}^n$ is equal to the span of the corresponding first $n$ eigenfunctions for $q_1$.
\end{proof}

\begin{proof}[Proof of Theorem~\ref{th:uniqueness}]
We will give the proof in detail only for the Dirichlet case \eqref{finitetracedir}; it is elementary, though tedious, to work 
through the (very similar) details in the Neumann case \eqref{finitetraceneu}, and so we will briefly sketch the proof of the Neumann case afterwards. Supposing that equality holds in~\eqref{finitetracedir} for some $n\geq 2$ (the case $n=1$ being trivial), denoting by $\varphi_k$, $k=1,\ldots, n$ the first $n$ eigenfunctions associated with $q$, without loss of generality chosen and numbered so that the conclusion of Lemma~\ref{lemma:equality} holds, since $\varphi_k(x) \neq 0$ almost everywhere in $(0,\pi)$, for each $k\geq 1$ we may write
\begin{displaymath}
	q(x) = \lambda_k(q)+\frac{\varphi_k''(x)}{\varphi_k(x)},
\end{displaymath}
which is valid pointwise almost everywhere. In particular, this means that for all $j,k=1,\ldots,n$ and almost all $x\in (0,\pi)$,
\begin{equation}
\label{eigenfndiff}
	\frac{\varphi_k''(x)}{\varphi_k(x)}=C_{k,j} +\frac{\varphi_j''(x)}{\varphi_j(x)},
\end{equation}
where $\R \ni C_{k,j}=\lambda_k(q)-\lambda_j(q)$. Since by Lemma~\ref{lemma:equality} we have
\begin{displaymath}
	\varphi_k(x)=\sum_{m=1}^n a_m^k \sin(mx)
\end{displaymath}
for appropriate $a_m^k \in \R$, after rearranging, we may rewrite \eqref{eigenfndiff} explicitly as
\begin{multline*}
\left[\sum_{m=1}^n m^2 a_m^k \sin(mx)\right]\left[\sum_{l=1}^n a_l^j \sin(lx)\right]
= \\ -C_{k,j}\left[\sum_{m=1}^n a_m^k \sin(mx)\right]\left[\sum_{l=1}^n a_l^j \sin(lx)\right]+
\left[\sum_{m=1}^n a_m^k \sin(mx)\right]\left[\sum_{l=1}^n l^2 a_l^j \sin(lx)\right]
\end{multline*}
which in turn may be rewritten as
\begin{multline}
\label{cosequality}
\sum_{l,m=1}^n m^2 a_m^k a_l^j \left[\cos(l-m)x -\cos(l+m)x \right]= -C_{k,j}
\sum_{l,m=1}^n a_m^k a_l^j \, \cdot \\ \cdot \left[\cos(l-m)x -\cos(l+m)x \right]+
\sum_{l,m=1}^n l^2 a_m^k a_l^j \left[\cos(l-m)x -\cos(l+m)x \right].
\end{multline}
We observe that \eqref{cosequality} is a sum of the form
\begin{displaymath}
	\sum_{p=0}^{2n} b_p \cos(px)=0,
\end{displaymath}
where the coefficients $b_p$ are obtained by summing all the relevant coefficients of $\cos(l-m)x$ with $|l-m|=p$ and $\cos(l+m)x$ with $l+m=p$. Since this holds for almost every $x\in(0,\pi)$, orthogonality of the (finite) family $\cos(px)$ implies that $b_p=0$ for all $p\geq 0$; we will use this to show that no combination of coefficients $a_m^k$ other than $a_m^k = c \delta_{mk}$ (with $c$ a normalizing constant) can satisfy \eqref{eigenfndiff}, and hence no non-constant $q$ is possible.

To do so we make a particular choice of $j,k$: without loss of generality, we may assume there exist two distinct eigenfunctions
$\varphi_j$, $\varphi_k$ such that $a_n^j, a_n^k \neq 0$, that is, both have non-zero $L^2$-projection onto $\spn\{\sin(nx)\}$; otherwise,
\eqref{span} 
forces $\varphi_i(x)=\sin(nx)$ for some $1\leq i\leq n$, and by comparing the two eigenfunction equations of the form \eqref{stlibdd} that $\sin(nx)$
must therefore satisfy, a routine argument shows that $q$ must be constant. So let us assume we have our $\varphi_j$ and $\varphi_k$
and consider the equation for $b_{2n}$. That is, equating coefficients of $\cos(2nx)$ in \eqref{cosequality},
\begin{displaymath}
	n^2 a_n^k a_n^j = -C_{k,j} a_n^k a_n^j + n^2 a_n^k a_n^j.
\end{displaymath}
Since $a_n^j, a_n^k \neq 0$, this implies $C_{k,j}=0$, that is, $\lambda_k(q)=\lambda_j(q)$. We claim that in this case
\begin{equation}
\label{equalratios}
	\frac{a_m^k}{a_n^k} = \frac{a_m^j}{a_n^j}
\end{equation}
for all $m=1,\ldots,n-1$. Let us first show why this proves the theorem. Assuming \eqref{equalratios} holds, 
and writing $a_m^k = c_m a_n^k$, $a_m^j = c_m a_n^j$ for $c_m \in \R$,
\begin{displaymath}
\sum_{m=1}^n c_m^2(a_n^k)^2 = \sum_{m=1}^n (a_m^k)^2 = \|\varphi_k\|_2^2 = \|\varphi_j\|_2^2
= \sum_{m=1}^n (a_m^j)^2 = \sum_{m=1}^n c_m^2(a_n^j)^2,
\end{displaymath}
implying $a_n^k = a_n^j$. \eqref{equalratios} now implies inductively that $a_m^k=a_m^j$ for all $m=1,\ldots,n$, that is, $\varphi_k = \varphi_j$, 
contradicting our assumption that $\varphi_k$ and $\varphi_j$ were two distinct eigenfunctions with non-trivial projection onto $\spn\{\sin(nx)\}$.
The only possibility is therefore that $\sin(nx)$ is itself an eigenfunction associated with $q(x)$, which implies $q$ is constant, as can be seen directly from the equation \eqref{stlibdd}.

It remains to prove \eqref{equalratios}. We will proceed by induction on $p$ from $2n$ down to $n+1$, equating the coefficients of $\cos(l+m)x$, 
$l+m=p$ in~\eqref{cosequality} in order to obtain \eqref{equalratios} for $a_{p-n}^k$, $a_{p-n}^j$. Observe that for each $p=n+1,\ldots,2n$,
\eqref{equalratios} reduces in this case to
\begin{equation}
\label{reduceddiff}
\sum (m^2-l^2)a_m^k a_l^j = 0,
\end{equation}
where the sum is over all $l,m=1,\ldots,n$ such that $l+m=p$. When $p=2n-1$, this says
\begin{displaymath}
	(n-1)^2 a_{n-1}^k a_n^j+n^2 a_n^k a_{n-1}^j = n^2 a_{n-1}^k a_n^j+(n-1)^2 a_n^k a_{n-1}^j,
\end{displaymath}
which upon rearrangement gives \eqref{equalratios} for $n-1$. Suppose now that \eqref{equalratios} holds for $p=2n-1,\ldots,n+i+1$, $i \geq 1$. 
Taking\eqref{reduceddiff} for $p=n+i$ and dividing through by $a_n^k a_n^j$, we have
\begin{displaymath}
\sum_{m=i}^n m^2 \frac{a_m^k}{a_n^k}\frac{a_{n+i-m}^j}{a_n^j}=
\sum_{l=i}^n l^2\frac{a_{n+i-l}^k}{a_n^k}\frac{a_l^j}{a_n^j}.
\end{displaymath}
Using the induction hypothesis that \eqref{equalratios} holds for $m=i+1,\ldots,n$, we see we can cancel all but the first terms in the above
equality, 
leaving
\begin{displaymath}
i^2 \frac{a_i^k}{a_n^k}\frac{a_n^j}{a_n^j} = i^2 \frac{a_n^k}{a_n^k}\frac{a_i^j}{a_n^j}.
\end{displaymath}
Hence \eqref{equalratios} holds for $i$, proving our claim.

Finally, let us remark that in the case of Neumann boundary conditions \eqref{finitetraceneu}, \eqref{eigenfndiff} is unchanged, while the expression for $\varphi_k$ is now
\begin{displaymath}
	\varphi_k(x) = \sum_{m=0}^n a_m^k \cos(mx),
\end{displaymath}
meaning \eqref{cosequality} is the same, only with expressions of the form $[\cos(l-m)x+\cos(l+m)x]$ replacing $[\cos(l-m)x -\cos(l+m)x]$, and with 
summation from $m=0$ to $n$ rather than $1$ to $n$, meaning our induction proceeds down to $p=n$. Otherwise, the argument is the same.
\end{proof}

\subsection*{Acknowledgements}
 J.B.K. was supported by a grant within the scope of the Funda\c c\~{a}o para a Ci\^{e}ncia e Tecnologia's project PTDC/\ MAT/\ 101007/2008 
 and by a fellowship of the Alexander von Humboldt Foundation, Germany. P.F. was partially supported by FCT's project 
 PTDC/\ MAT-CAL/\ 4334/2014.

\end{document}